\documentclass[11pt]{amsart}
\makeatletter
\let\th@theorem\th@plain
\makeatother
\usepackage{amsfonts, amsmath, amssymb, amsthm}
\usepackage{anysize}
\usepackage{mathtools}
\marginsize{2cm}{2cm}{1.5cm}{1.5cm}

\usepackage[utf8]{inputenc}
\usepackage[english]{babel}
\usepackage{cmap}
\usepackage{enumerate}

\usepackage{xcolor}
\usepackage{hyperref}
\hypersetup{
	colorlinks   = true, 
	urlcolor     = blue, 
	linkcolor    = blue, 
	citecolor    = red 
}
\newcommand{\Href}[2]{\hyperref[#2]{#1~\ref{#2}}}

\usepackage{cancel}
\usepackage{tikz}
\usetikzlibrary{arrows.meta}
\usepackage{placeins}

\theoremstyle{theorem}
\newtheorem{thm}{Theorem}[section]
\newtheorem{prp}{Proposition}[section]
\newtheorem{lem}{Lemma}[section]
\newtheorem{cor}{Corollary}[section]
\newtheorem{conj}{Conjecture}[section]

\theoremstyle{definition}
\newtheorem{dfn}{Definition}[section]
\newtheorem{rem}[dfn]{Remark}

\newcommand{\norm}[1]{\left\|#1\right\|}

\newcommand{\enorm}[1]{\left|#1\right|}

\DeclareMathOperator{\tr}{\mathrm{trace}}

\providecommand{\parenth}[1]{\left(#1\right)}%
\providecommand{\braces}[1]{\left\{#1\right\}}%

\newcommand{\conv}{\mathrm{conv}}%
\newcommand{\cof}[1]{\conv \! \braces{#1}}
\newcommand{\iprod}[2]{\left\langle#1, #2\right\rangle}%

\def\R{{\mathbb R}}%
\def\C{{\mathbb C}}%

\def\closure{\mathop{\rm cl}}%
\def\phi{\varphi}
\def\epsilon{\varepsilon}

\newcommand{\id}{\mathrm{Id}}
\newcommand{\prooff}{\noindent {\bf Proof.}\\}%
\newcommand{\bbox}{\par\noindent\ensuremath{\Box}\par\noindent}%

\newcommand{\di}{\,\mathrm{d}}

\newcommand{\romenc}[1]{\uppercase\expandafter{\romannumeral #1\relax}}%

\providecommand{\abs}[1]{\lvert#1\rvert}%
\def\polar{\circ}

\newcommand{\ball}[1]{\mathbf{B}^{#1}}

\usepackage[expansion=false]{microtype}
\hypersetup{pdftitle={Halfspaces containing a ball},pdfauthor={Grigory Ivanov},
 pdfsubject={Gaussian orthant probabilities and a sharp radius bound}}
\newcommand{\one}{\mathbf{1}}
\newcommand{\vertices}{\operatorname{vert}}
\newcommand{\HS}{\mathrm{HS}}
\numberwithin{equation}{section}

\title[Zong's problem on half-spaces]
{Zong's problem on $2d$ half-spaces containing the unit ball}

\subjclass[2020]{Primary 52A40; Secondary 52B11, 60E15}

\keywords{Half-spaces, convex polytopes, spherical coverings,
Gaussian orthant probabilities, tight frames}

\hypersetup{
    pdftitle={Zong's problem on 2d half-spaces containing the unit ball}
}
\author{Grigory Ivanov}
\date{}

\begin{document}
\begin{abstract}
We prove that the intersection of $2d$ half-spaces containing the Euclidean
unit ball in $\mathbb{R}^d$ contains a point of norm at least $\sqrt{d}$.
Equality in the radius bound characterizes the cube up to orthogonal
transformations. The proof gives a lower bound for the average squared
norm of the vertices, weighted by the Gaussian measures of their normal
cones. This estimate follows from an inequality for the Gaussian
measures of a simplicial cone and its projected facets, which in turn
reduces to a quadratic estimate for the
Gaussian centroid of a simplicial cone. 

We also discuss a quadratic extension of the problem and give a bound for it using an independent spectral argument. This bound implies that
the intersection of any $m > d$ half-spaces containing the unit ball
contains a point of norm at least $\sqrt{m}\,\ln\frac{m}{m-d}$.
\end{abstract}
\maketitle

\section{Introduction}

Write $[m] = \{1, \dots, m\}$, let $\enorm{\cdot}$ denote the Euclidean norm,
and let $\ball{d}$ be the Euclidean unit ball in $\R^d$.
We write $x[i]$ for the $i$th coordinate of a vector $x$.
Our main result is the following.

\begin{thm}\label{thm:halfspaces}
Let $H_1, \dots, H_{2d}$ be closed half-spaces containing $\ball{d}$, and put
$P = \bigcap_{i \in [2d]} H_i$. Then $P$ contains a point $x$ with
\[
 \enorm{x} \geq \sqrt{d}.
\]
Moreover, if $\max_{x \in P} \enorm{x} = \sqrt{d}$, then $P$ is an orthogonal
image of $[-1, 1]^d$.
\end{thm}

This resolves a problem posed by Chuanming Zong
\cite[Problem~4]{dalla2000blocking}. The problem also appears in Zong's
survey on unit cubes \cite[Problem~8.1, p.~206]{zong2005known} and in
Schmitt and Ziegler's collection of geometry problems
\cite[Problem~9]{schmitt2012ten}. L.~Fejes T\'oth solved the
three-dimensional case \cite{fejestoth1972lagerungen}, and
Dalla, Larman, Mani-Levitska, and Zong proved the four-dimensional case
\cite[Theorem~6]{dalla2000blocking}.
By polarity, the problem is equivalent to finding an optimal covering of
$\partial \ball{d}$ by $2d$ equal spherical caps
(see also \cite[Chapter~6]{boroczky2004finite}).
In particular, \Href{Theorem}{thm:halfspaces} gives
\[
 \min_{u \in \partial \ball{d}} \max_{i \in [2d]} \iprod{a_i}{u}
 \leq \frac{1}{\sqrt{d}}
 \qquad \text{for } a_1, \dots, a_{2d} \in \partial \ball{d}.
\]
Equality is attained when the $a_i$ are the vertices of the standard
cross-polytope.

The proof uses an average over the vertices of a polytope, with weights
given by the Gaussian measures of their normal cones.
Let $\gamma_d$ be the standard Gaussian probability measure on $\R^d$,
with density $(2 \pi)^{-d/2} e^{-\enorm{x}^2/2}$.
For a polytope $P$, we write $\vertices(P)$ for its set of vertices.

\begin{dfn}\label{dfn:gaussian-vertex-measure}
Let $P \subset \R^d$ be a bounded full-dimensional polytope.
For $v \in \vertices(P)$, its normal cone is
\[
 N_P(v) = \{p \in \R^d: \iprod{p}{x - v} \leq 0 \text{ for every } x \in P\}.
\]
The \emph{Gaussian vertex measure} of $P$ assigns weight
$\gamma_d(N_P(v))$ to $v$. Its second moment operator is
\[
 A_\gamma(P) = \sum_{v \in \vertices(P)} \gamma_d(N_P(v))\, v \otimes v,
\]
where $(v \otimes v)x = \iprod{v}{x} v$.
\end{dfn}

The normal cones partition $\R^d$ up to their boundaries, so the weights
sum to one. Equivalently, for a vector $p$ sampled from $\gamma_d$, let
$v_P(p)$ be a vertex maximizing $\iprod{p}{x}$ over $P$. The maximizing
vertex is unique for $\gamma_d$-almost every $p$, and
\begin{equation} \label{eq:gaussian-moment}
 A_\gamma(P) = \int_{\R^d} v_P(p) \otimes v_P(p) \di \gamma_d(p).
\end{equation}
Thus, $\tr A_\gamma(P)$ is the expected squared norm of the maximizing
vertex. When all facets are tangent to the unit ball, the following
bound depends only on their number and the dimension.

\begin{thm}\label{thm:gaussian-vertices}
Let $a_1, \dots, a_m \in \partial \ball{d}$ be distinct, and assume that
\[
 P = \{x \in \R^d: \iprod{a_i}{x} \leq 1, \ i \in [m]\}
\]
is bounded. Then
\begin{equation} \label{eq:average}
 \tr A_\gamma(P)
 = \sum_{v \in \vertices(P)} \gamma_d(N_P(v)) \enorm{v}^2
 \geq d + \frac{\pi}{4}(2d - m).
\end{equation}
Equality holds if and only if $m = 2d$ and $P$ is an orthogonal image
of $[-1, 1]^d$.
\end{thm}

The geometric argument follows the angle-sum approach proposed in
\cite[Remark~2]{dalla2000blocking} and uses the Gaussian cone weights
from our work on projections of the cross-polytope
\cite{ivanov2026maximalvolumeprojectionscrosspolytope}.
We triangulate the boundary of the polar polytope without adding
vertices. For each vertex $a$, we project the cones over the incident
simplices onto $a^\perp$, which is parallel to the hyperplane tangent
to the unit sphere at $a$. We used a related analysis of the local
structure at a vertex to study projections of the cross-polytope
in \cite{ivanov2021volume}.
The projected cones partition $a^\perp$ up to their boundaries, so
their standard Gaussian measures sum to one at each vertex. Summing
over the vertices of the polar gives $m$. Combining this identity with
the partition of $\R^d$ by the original cones reduces
\Href{Theorem}{thm:gaussian-vertices} to an inequality for the Gaussian
measures of a simplicial cone and its projected facets.

Let $G$ be a symmetric positive definite $d\times d$ matrix. Choose
$B$ with $B^TB=G$, and denote its columns by $b_1,\dots,b_d$.
Write $\R_+^d=(0,\infty)^d$, and put $C=B\R_+^d$. For each $i\in[d]$,
let $F_i$ be the facet of $\closure C$ opposite the ray generated by
$b_i$, and let $\pi_i$ be orthogonal projection onto $b_i^\perp$, so that
\[
 F_i=B\{x\in[0,\infty)^d:x[i]=0\},\qquad \pi_i(z)=z-\frac{\iprod{z}{b_i}}{G_{ii}}b_i.
\]
Define
\begin{equation}\label{eq:gaussian-ratio}
 \Psi_G=\frac{\sum_{i\in[d]}\gamma_{d-1}(\pi_iF_i)}{\gamma_d(C)},\qquad \delta_G=(\sqrt{G_{ii}})_{i\in[d]}.
\end{equation}
Here $\gamma_{d-1}$ denotes standard Gaussian measure on the indicated
hyperplane $b_i^\perp$. Thus, $\Psi_G$ compares the sum of the Gaussian
measures of the projected facets with the Gaussian measure of the cone.
Any two choices of $B$ differ by an orthogonal transformation, so
$\Psi_G$ depends only on $G$. For $d=1$, we use
$\gamma_0(\{0\})=1$, giving $\Psi_G=2$.

\begin{thm}\label{thm:gaussian}
For every symmetric positive definite $G\in\R^{d\times d}$,
\begin{equation}\label{eq:gaussian}
 \Psi_G+\frac4\pi\iprod{G^{-1}\delta_G}{\delta_G}\ge\parenth{2+\frac4\pi}d.
\end{equation}
Equality holds if and only if $G$ is diagonal.
\end{thm}

We estimate each projected facet measure by integrating the Gaussian
density along lines parallel to the opposite generator. Reflection
and the Cauchy--Schwarz inequality give a quadratic lower bound in
terms of boundary integrals, which we express through the Gaussian
centroid in orthant coordinates. The remaining estimate is quadratic
in that centroid.
For background on the Gaussian calculations, see
\cite[Section~2.5]{anderson2003introduction} and
\cite[Section~3.3.2 and Lemmas~7.2.3, 7.2.5]{vershynin2018high}.
Moments of Gaussian distributions restricted to an orthant were studied
by Tallis~\cite{tallis1961moment}.

We denote by $\one\in\R^d$ the vector with all coordinates equal to one,
and by $\id_d$ the identity operator on $\R^d$.
For matrices, $\iprod{A}{B}=\tr(A^TB)$ is the Hilbert--Schmidt inner
product on $\R^{d\times d}$, with norm $\norm{\cdot}_{\HS}$.

In \Href{Section}{sec:halfspace-proof}, we deduce
\Href{Theorem}{thm:halfspaces} from \Href{Theorem}{thm:gaussian-vertices}
and reduce the Gaussian vertex bound to \Href{Theorem}{thm:gaussian}.
\Href{Section}{sec:gaussian-proof} proves the Gaussian inequality through
the quadratic centroid estimate of \Href{Theorem}{thm:centroid-estimate}.
\Href{Section}{sec:quadratic} discusses a quadratic extension and open
questions suggested by the proof.
\Href{Section}{sec:spectral} adapts the method of Ball and
Prodromou~\cite{Ball2009} to one-sided inequalities. This argument grew
out of our initial attempt to solve Zong's problem. It gives a bound
for the quadratic formulation and yields the radius bound
$\sqrt{2}\ln 2\,\sqrt d>0.98\sqrt d$.
For a related conjecture on sparse approximation, see
\cite[Conjecture~5]{almendra2022quantitative}.

\subsection*{Use of AI}

After the geometric reduction of Zong's problem to
\Href{Theorem}{thm:gaussian} had been established, ChatGPT provided an
initial proof of the Gaussian inequality. The article presents an
alternative proof developed using
ideas of the author and Fedor Nazarov. ChatGPT also assisted with
formalizing the arguments and proofreading the manuscript.

\section{Proof of the half-space theorem}\label{sec:halfspace-proof}

\subsection{Deduction of \Href{Theorem}{thm:halfspaces}}

\begin{proof}[Proof of \Href{Theorem}{thm:halfspaces}]
If $P$ is unbounded, the conclusion is immediate. Otherwise, move each
bounding hyperplane inward until it is tangent to $\ball{d}$. The
resulting polytope $P_0 \subseteq P$ is bounded and has $m \leq 2d$
distinct facet normals. \Href{Theorem}{thm:gaussian-vertices} gives
\[
 \max_{x \in P} \enorm{x}^2 \geq \tr A_\gamma(P_0)
 \geq d + \frac{\pi}{4}(2d - m) \geq d.
\]
The first inequality holds because every vertex of $P_0$ belongs to $P$
and the Gaussian weights sum to one.

If $\max_{x \in P} \enorm{x} = \sqrt{d}$, all inequalities are equalities.
Thus, $m = 2d$ and $P_0$ is a rotated cube. Its normals are the pairs
$\pm e_i$ in some orthonormal basis. As there were only $2d$ original half-spaces,
each of these normals occurs exactly once. Moving any one of their
hyperplanes outward increases the norm of a corner. Therefore, all the
original hyperplanes were already tangent, and $P = P_0$.
The converse follows by inspecting the vertices of the cube.
\end{proof}

\subsection{Proof of the Gaussian vertex bound}

\begin{proof}[Proof of \Href{Theorem}{thm:gaussian-vertices}]
For $d=1$, boundedness forces $m=2$ and $P=[-1,1]$, which gives the result.
Assume $d\ge2$. The polar body of $P$ is
\[
 P^\polar=\{y\in\R^d:\iprod{y}{x}\le1\text{ for every }x\in P\}=\cof{a_i:i\in[m]}.
\]
Its vertices are precisely $a_1,\dots,a_m$, since these are distinct
unit vectors. Triangulate the boundary of $P^\polar$ without adding vertices,
and let $\mathcal T$ be the set of $(d-1)$-dimensional simplices in
this triangulation.

For each $\Delta\in\mathcal T$, choose an ordering of its vertices
and let $B_\Delta$ be the matrix with these vertices as columns.
They are linearly independent because they are affinely independent and
lie in a supporting hyperplane of $P^\polar$ that misses the origin.
Let $v_\Delta$ be the unique point satisfying
\begin{equation}\label{eq:simplex-vertex}
 \iprod{a}{v_\Delta}=1\qquad\text{for every }a\in\vertices(\Delta).
\end{equation}
The facet of $P^\polar$ containing $\Delta$ has a supporting
hyperplane $\iprod{z}{x}=1$ with $z\in P$. The $d$ vertices of
$\Delta$ force $z=v_\Delta$. The corresponding $d$ constraints of $P$
are active at $v_\Delta$ and have linearly independent normals, so
$v_\Delta$ is a vertex of $P$. In particular, all simplices in the
same facet of $P^\polar$ give the same vertex of $P$.

The cones $B_\Delta\R_+^d$, $\Delta\in\mathcal T$, partition $\R^d$
up to their boundaries. The cones associated with the simplices in a
given facet of $P^\polar$ partition the normal cone of the corresponding
vertex of $P$.
Consequently,
\begin{equation}\label{eq:cone-partition}
 \sum_{\Delta\in\mathcal T}\gamma_d(B_\Delta\R_+^d)=1,\qquad \tr A_\gamma(P)=\sum_{\Delta\in\mathcal T}\gamma_d(B_\Delta\R_+^d)\enorm{v_\Delta}^2.
\end{equation}

We next examine the projected cones around a vertex $a$ of $P^\polar$.
Let $\pi_a$ be orthogonal projection onto $a^\perp$. For every
$\Delta\in\mathcal T$ containing $a$, put
\[
 D_{\Delta,a}=\operatorname{pos}\{\pi_a(b):b\in\vertices(\Delta)\setminus\{a\}\},\qquad \pi_a(b)=b-\iprod{b}{a}a.
\]
Here $\operatorname{pos}$ denotes the cone of nonnegative linear
combinations. The cones $D_{\Delta,a}$ partition $a^\perp$ up to their
boundaries. To see this, consider $C_a=\operatorname{pos}(P^\polar-a)$.
Since $0\in\operatorname{int}P^\polar$, the vector $-a$ lies in
$\operatorname{int}C_a$. Also, $\iprod{a}{z}<0$ for every nonzero
$z\in C_a$, because all other vertices of $P^\polar$ are unit vectors
distinct from $a$. It follows that every line parallel to $a$ meets
$\partial C_a$ in exactly one point. Project the subdivision of
$\partial C_a$ induced by the simplices incident with $a$, and use
$\pi_a(b-a)=\pi_a(b)$. This gives the claimed partition.
Hence, for each fixed vertex $a$,
\begin{equation}\label{eq:projected-at-vertex}
 \sum_{\substack{\Delta\in\mathcal T\\a\in\vertices(\Delta)}}\gamma_{d-1}(D_{\Delta,a})=1,
\end{equation}
where $\gamma_{d-1}$ denotes standard Gaussian measure on $a^\perp$.

Summing \eqref{eq:projected-at-vertex} over the vertices of $P^\polar$
and grouping the terms by simplex gives
\begin{equation}\label{eq:incidence-sum}
 \sum_{\Delta\in\mathcal T}\sum_{a\in\vertices(\Delta)}\gamma_{d-1}(D_{\Delta,a})=\sum_{a\in\vertices(P^\polar)}\sum_{\substack{\Delta\in\mathcal T\\a\in\vertices(\Delta)}}\gamma_{d-1}(D_{\Delta,a})=m.
\end{equation}

For each $\Delta\in\mathcal T$, put $G_\Delta=B_\Delta^TB_\Delta$.
For a vertex $a$ of $\Delta$, the cone $D_{\Delta,a}$ is the projection
onto $a^\perp$ of the facet of $\closure(B_\Delta\R_+^d)$ opposite the
ray generated by $a$. The definition \eqref{eq:gaussian-ratio} therefore
gives
\begin{equation}\label{eq:simplex-psi}
 \sum_{a\in\vertices(\Delta)}\gamma_{d-1}(D_{\Delta,a})=\gamma_d(B_\Delta\R_+^d)\Psi_{G_\Delta}.
\end{equation}
Substituting this identity into \eqref{eq:incidence-sum}, we obtain
\begin{equation}\label{eq:projected-partition}
 \sum_{\Delta\in\mathcal T}\gamma_d(B_\Delta\R_+^d)\Psi_{G_\Delta}=m.
\end{equation}

Each column of $B_\Delta$ has norm one, so $\delta_{G_\Delta}=\one$.
Moreover, \eqref{eq:simplex-vertex} gives $B_\Delta^Tv_\Delta=\one$,
and hence
\[
 \iprod{G_\Delta^{-1}\delta_{G_\Delta}}{\delta_{G_\Delta}}=\iprod{G_\Delta^{-1}\one}{\one}=\enorm{v_\Delta}^2.
\]
Apply \Href{Theorem}{thm:gaussian} to each $G_\Delta$, multiply by
$\gamma_d(B_\Delta\R_+^d)$, and sum over $\Delta\in\mathcal T$.
The identities \eqref{eq:cone-partition} and
\eqref{eq:projected-partition} yield
\[
 m+\frac4\pi\tr A_\gamma(P)\ge\parenth{2+\frac4\pi}d,
\]
which is \eqref{eq:average}.

If equality holds, every application of \Href{Theorem}{thm:gaussian}
is an equality, since every $\gamma_d(B_\Delta\R_+^d)$ is positive.
Thus, $G_\Delta=\id_d$ for all $\Delta\in\mathcal T$, and every
boundary simplex has an orthonormal vertex set. Across a common
$(d-2)$-dimensional face, the two missing unit vectors are distinct
and orthogonal to the same $(d-1)$-dimensional subspace, so they are
opposite. The adjacency graph of the boundary triangulation is
connected. Starting with one simplex and following adjacent simplices,
we conclude that all vertices of $P^\polar$ belong to
$\{\pm e_i:i\in[d]\}$ after a rotation.
Since $0\in\operatorname{int}P^\polar$, all these vertices occur. Hence,
$m=2d$, $P^\polar$ is a regular cross-polytope, and $P$ is a cube.
Conversely, all vertices of $[-1,1]^d$ have squared norm $d$, giving
equality in \eqref{eq:average}.
\end{proof}

\section{Proof of the Gaussian inequality}\label{sec:gaussian-proof}

In this section, we prove \Href{Theorem}{thm:gaussian}.
The proof consists of three steps.

First, we bound $\Psi_G$ from below in terms of the centroid of
the Gaussian density $f_G$, defined below, restricted to the positive
orthant. This reduces the desired estimate to a quadratic inequality
for that centroid.

Second, we use the geometry underlying $\Psi_G$: the Gaussian
measures of a simplicial cone and its $(d-1)$-dimensional projected
facets. The corresponding change of variables identifies the
orthant centroid with the Gaussian centroid of the cone. We express
the left-hand side of the centroid inequality in terms of squared
lengths of vectors and their projections. Completing a square then
reduces the problem to coordinate inequalities that are affine in
one additional geometric parameter.

Third, we bound this additional parameter from below and above. The key is that
the geometric parameters in these coordinate inequalities have
explicit expressions as averages against $f_G$. The restriction of our density to each coordinate
ray is a one-dimensional truncated
Gaussian. Averaging an inequality for its centroid gives the lower
bound, while log-concavity gives the upper bound. The affine
dependence then leaves only the two endpoint values to check.

For ease of reading, we explain each reduction as it arises,
indicating which statement it suffices to prove. At the end of
the section, we assemble these reductions and estimates into
complete proofs, including the equality conditions.

We begin by defining the centroid.
For a positive definite $G$, consider the Gaussian density
\begin{equation}\label{eq:gaussian-density}
 f_G(x)=\frac{\sqrt{\det G}}{(2\pi)^{d/2}}\exp\parenth{-\frac12\iprod{Gx}{x}}.
\end{equation}
This is the centered Gaussian density with covariance matrix $G^{-1}$.
We denote its centroid in the positive orthant by
\begin{equation}\label{eq:orthant-centroid}
 \mu=\frac{\int_{\R_+^d}x f_G(x)\di x}{\int_{\R_+^d}f_G}.
\end{equation}
We use square brackets for vector coordinates: $\mu[i]$ and
$(G\mu)[i]$ denote the $i$th coordinates of $\mu$ and $G\mu$,
respectively.
\Href{Theorem}{thm:gaussian} follows from the following quadratic
centroid inequality.

\begin{thm}[The quadratic centroid inequality]\label{thm:centroid-estimate}
For every positive definite $G$,
\begin{equation}\label{eq:centroid-quadratic}
 \frac\pi2\sum_{i\in[d]}\parenth{\frac{(G\mu)[i]}{\sqrt{G_{ii}}}}^2
 +\frac4\pi\iprod{G^{-1}\delta_G}{\delta_G}
 \ge\parenth{1+\frac4\pi}d.
\end{equation}
Equality holds if and only if $G$ is diagonal.
\end{thm}

We assume $d\ge2$ in the coordinate calculations below. The proofs of
both theorems treat $d=1$ directly.

\subsection{From boundary integrals to the centroid}

To reduce \Href{Theorem}{thm:gaussian} to
\Href{Theorem}{thm:centroid-estimate}, we first bound the Gaussian
measures of the projected facets by integrals of $f_G$ over the
coordinate faces of the orthant. We then express those integrals
in terms of $\mu$.
For each $i$, write $t=x[i]$ and $y=(x[j])_{j\ne i}$, retaining
the original indices $j\ne i$ for the coordinates of $y$.
In $f_G(t,y)$, the variable $t$ occupies position $i$.
Expanding the quadratic form in the exponent gives
\begin{equation}\label{eq:coordinate-slice}
 f_G(t,y)=f_G(0,y)\exp\parenth{-\frac{G_{ii}t^2}{2}-t\sum_{j\ne i}G_{ij}y[j]}.
\end{equation}

\begin{lem}\label{lem:projection}
For every positive definite $G$,
\begin{equation}\label{eq:projection-boundary}
 \Psi_G\ge d+\frac\pi2\sum_{i\in[d]}\parenth{\frac{\int_{\R_+^{d-1}}f_G(0,y)\di y}{\sqrt{G_{ii}}\,\int_{\R_+^d}f_G}}^2.
\end{equation}
\end{lem}

\begin{proof}
Choose $B$ with $B^TB=G$ and columns $b_1,\dots,b_d$.
For each $i\in[d]$, recall that
\[
 F_i=B\{x\in[0,\infty)^d:x[i]=0\}
\]
is the facet of $\closure(B\R_+^d)$ opposite the ray generated
by $b_i$, and $\pi_i$ is orthogonal projection onto $b_i^\perp$.
Put
\[
 U_i=\{x\in\R^d:x[j]>0\text{ for }j\ne i\},
\]
so that coordinate $i$ is unrestricted. Up to their boundaries,
$BU_i=F_i+\R b_i=\pi_iF_i+\R b_i$. Standard Gaussian measure
factors over the orthogonal subspaces $b_i^\perp$ and $\R b_i$,
and the second factor has total mass one. Hence,
\[
 \gamma_{d-1}(\pi_iF_i)=\gamma_d(BU_i)=\int_{U_i}f_G,
\]
where the last equality follows from the change of variables $z=Bx$.
Thus, by \eqref{eq:gaussian-ratio},
\[
 \Psi_G=\frac{\sum_{i\in[d]}\int_{U_i}f_G}{\int_{\R_+^d}f_G}.
\]

Let $D_i$ reverse the sign of coordinate $i$. By the
Cauchy--Schwarz inequality,
\[
 \frac{\int_{\R_+^d}f_G(D_i x)\di x}{\int_{\R_+^d}f_G}
 \ge\parenth{\frac{\int_{\R_+^d}\sqrt{f_G(x)f_G(D_i x)}\di x}{\int_{\R_+^d}f_G}}^2
 =\frac\pi{2G_{ii}}\parenth{\frac{\int_{\R_+^{d-1}}f_G(0,y)\di y}{\int_{\R_+^d}f_G}}^2.
\]
For the equality, use the slice formula \eqref{eq:coordinate-slice} to write
$\sqrt{f_G(t,y)f_G(-t,y)}=f_G(0,y)e^{-G_{ii}t^2/2}$.
Integrating first over $t>0$ gives $\sqrt{\pi/(2G_{ii})}f_G(0,y)$,
and integrating over $y$ gives the boundary integral.
Up to a set of Gaussian measure zero, $U_i$ is the disjoint union of
$\R_+^d$ and $D_i\R_+^d$. Therefore, summing over $i$ yields
\[
 \Psi_G\ge d+\frac\pi2\sum_{i\in[d]}\parenth{\frac{\int_{\R_+^{d-1}}f_G(0,y)\di y}{\sqrt{G_{ii}}\,\int_{\R_+^d}f_G}}^2.
\]
\end{proof}

We now express the boundary integrals in terms of $\mu$.

\begin{lem}\label{lem:boundary}
For every positive definite $G$ and every $i\in[d]$,
\begin{equation}\label{eq:boundary-centroid}
 (G\mu)[i]=
 \frac{\int_{\R_+^{d-1}}f_G(0,y)\di y}{\int_{\R_+^d}f_G}>0.
\end{equation}
\end{lem}

\begin{proof}
Differentiating the slice formula \eqref{eq:coordinate-slice} gives
\[
 \partial_t f_G(t,y)=-\parenth{G_{ii}t+\sum_{j\ne i}G_{ij}y[j]}f_G(t,y).
\]
The expression in parentheses is precisely $(Gx)[i]$ when $x[i]=t$
and $x[j]=y[j]$ for $j\ne i$. Integration in $t$ therefore gives
\[
 f_G(0,y)=-\int_0^\infty\partial_t f_G(t,y)\di t
          =\int_0^\infty(Gx)[i]f_G(t,y)\di t.
\]
Integrating over $y$ and applying Fubini's theorem,
we obtain
\[
 \int_{\R_+^{d-1}}f_G(0,y)\di y
 =\int_{\R_+^d}(Gx)[i]f_G(x)\di x
 =(G\mu)[i]\int_{\R_+^d}f_G.
\]
The boundary integral is positive, proving the claim.
\end{proof}

\begin{cor}\label{cor:projection-centroid}
For every positive definite $G$,
\begin{equation}\label{eq:projection-centroid}
 \Psi_G\ge d+\frac\pi2\sum_{i\in[d]}\parenth{\frac{(G\mu)[i]}{\sqrt{G_{ii}}}}^2.
\end{equation}
\end{cor}

\begin{proof}
Substitute the identity from \Href{Lemma}{lem:boundary} into
\eqref{eq:projection-boundary}.
\end{proof}

Substituting the lower bound from
\Href{Corollary}{cor:projection-centroid} for $\Psi_G$ in
\eqref{eq:gaussian} shows that \Href{Theorem}{thm:gaussian} follows
from \Href{Theorem}{thm:centroid-estimate}. We therefore turn to
the centroid inequality.

\subsection{Geometric interpretation and the coordinate reduction}

We now reformulate the centroid inequality in geometric terms.
This will suggest a square to complete and identify the quantities
needed to estimate the remainder.

Choose $B$ with $B^TB=G$ and put $C=B\R_+^d$. The centroid of
the standard Gaussian measure restricted to $C$ is
\[
 c=\frac1{\gamma_d(C)}\int_C z\di\gamma_d(z)=B\mu.
\]
Let $u_i=b_i/\sqrt{G_{ii}}$ be the unit generators of $C$.
The first term in \eqref{eq:centroid-quadratic} is determined by
the scalar projections of $c$ onto these directions:
\begin{equation}\label{eq:coordinate-projection}
 h_i=\iprod{c}{u_i}=\frac{(G\mu)[i]}{\sqrt{G_{ii}}}>0.
\end{equation}
Here we used $B^TB=G$, and positivity follows from
\Href{Lemma}{lem:boundary}.

To interpret the second term, let $v$ be the common point of the
hyperplanes tangent to the unit ball at $u_1,\dots,u_d$.
The same calculation as in the geometric
reduction in the proof of \Href{Theorem}{thm:gaussian-vertices} gives
\begin{equation}\label{eq:tangent-norm}
 v=B^{-T}\delta_G,\qquad
 \enorm{v}^2=\iprod{B^{-1}v}{B^Tv}
 =\iprod{G^{-1}\delta_G}{\delta_G}.
\end{equation}
Consequently, \Href{Theorem}{thm:centroid-estimate} is equivalent
to the following geometric statement.

\begin{cor}\label{cor:centroid-geometric}
Let $C$ be a full-dimensional simplicial cone with unit generators
$u_1,\dots,u_d$, let $c$ be its Gaussian centroid, and let $v$
be the common point of the hyperplanes tangent to the unit ball
at these generators. Then
\begin{equation}\label{eq:centroid-quadratic-geometric}
 \frac\pi2\sum_{i\in[d]}\iprod{c}{u_i}^2+\frac4\pi\enorm{v}^2
 \ge\parenth{1+\frac4\pi}d.
\end{equation}
Equality holds if and only if the generators are mutually orthogonal.
\end{cor}

\begin{proof}
Use \eqref{eq:coordinate-projection} and \eqref{eq:tangent-norm}
in \eqref{eq:centroid-quadratic}. Every full-dimensional simplicial
cone admits such a representation, and diagonality of the Gram
matrix is equivalent to orthogonality of its generators.
\end{proof}

The first term is $\pi/2$ times the squared norm of the vector of
projections of $c$ onto the unit generators. For mutually orthogonal
generators, $v=\sum_i u_i$ and $c=\sqrt{2/\pi}\,v$, giving equality.
It is therefore natural to complete the square
$\enorm{c-\sqrt{2/\pi}\,v}^2$, which vanishes in these expected
equality cases. Expanding this square introduces
$\iprod{c}{v}$ and $\enorm{c}^2$. To express these terms as
coordinate sums, we introduce one further geometric quantity.

Fix $i\in[d]$, write $u=u_i$, and let $e$ be the
$i$th coordinate vector. The point $q=B(\mu-\mu[i] e)$ belongs
to the opposite facet $F_i$, and
\[
 c-q=\mu[i] Be=\sqrt{G_{ii}}\,\mu[i] u.
\]
Thus, the length from $q$ to $c$ along the ray direction $u$ is
\begin{equation}\label{eq:coordinate-length}
 \ell_i=\sqrt{G_{ii}}\,\mu[i]>0.
\end{equation}
Defining $\ell_i$ in this way for every $i\in[d]$, we obtain
\begin{equation}\label{eq:centroid-geometric-identities}
 c=\sum_{i\in[d]}\ell_i u_i,\qquad
 \enorm{c}^2=\sum_{i\in[d]}\ell_i h_i,\qquad
 \iprod{c}{v}=\sum_{i\in[d]}\ell_i.
\end{equation}
The first identity is the expansion of $B\mu$ in the columns of
$B$. Taking scalar products with $c$ and $v$ gives the other
two identities; the last uses the tangent-hyperplane property of $v$.
We can now complete the square:
\begin{equation}\label{eq:centroid-square}
 \frac\pi2\sum_{i\in[d]}h_i^2+\frac4\pi\enorm{v}^2
 =2\enorm{c-\sqrt{\frac2\pi}\,v}^2
 +\sum_{i\in[d]}\parenth{\frac\pi2 h_i^2+\parenth{4\sqrt{\frac2\pi}-2h_i}\ell_i}.
\end{equation}
It remains to bound each summand in the remainder. We shall prove
the following coordinate estimate.

\begin{lem}[The quadratic coordinate estimate]\label{lem:coordinate-quadratic}
For each $i\in[d]$,
\begin{equation}\label{eq:coordinate-quadratic}
 \frac\pi2 h_i^2+\parenth{4\sqrt{\frac2\pi}-2h_i}\ell_i
 \ge1+\frac4\pi.
\end{equation}
Equivalently, in terms of $G$ and the orthant centroid $\mu$,
\begin{equation}\label{eq:coordinate-quadratic-matrix}
 \frac\pi2\parenth{\frac{(G\mu)[i]}{\sqrt{G_{ii}}}}^2
 +4\sqrt{\frac2\pi}\sqrt{G_{ii}}\,\mu[i]-2\mu[i](G\mu)[i]
 \ge1+\frac4\pi.
\end{equation}
Equality holds if and only if $\ell_i=h_i=\sqrt{2/\pi}$, or,
equivalently, $G_{ij}=0$ for every $j\ne i$.
\end{lem}

For fixed $h_i>0$, the expression in \eqref{eq:coordinate-quadratic}
is affine in $\ell_i$. It is therefore enough to bound $\ell_i$
from below and above and evaluate the expression at the two
endpoints. The bounds we need are
\begin{equation}\label{eq:coordinate-segment}
 \sqrt{\frac\pi2}-\parenth{\frac\pi2-1}h_i\le\ell_i\le\frac1{h_i}.
\end{equation}
These are two linear constraints on $\ell_i$ when $h_i$ is fixed.
We now prove the lower bound by averaging a one-dimensional
Gaussian inequality over rays, and the upper bound by log-concavity.

\subsection{Averaging over coordinate rays}

For $s\in\R$, let
\[
 \lambda(s)=\frac{\int_s^\infty t e^{-t^2/2}\di t}{\int_s^\infty e^{-t^2/2}\di t}
 =\frac{e^{-s^2/2}}{\int_s^\infty e^{-t^2/2}\di t}.
\]
This is the centroid of the standard Gaussian measure restricted
to $(s,\infty)$, and $\lambda(s)-s$ is its distance from the
endpoint. The supporting line to $\lambda$ at zero gives the
one-dimensional estimate that we shall average.

\begin{lem}\label{lem:hazard}
For every $s\in\R$,
\begin{equation}\label{eq:hazard-centroid}
 \parenth{\lambda(s)-s}+\parenth{\frac\pi2-1}\lambda(s)\ge\sqrt{\frac\pi2},
\end{equation}
with equality if and only if $s=0$.
\end{lem}

\begin{proof}
The function $\lambda$ is strictly convex on $\R$
\cite{sampford1953some}. Its value at zero is $\sqrt{2/\pi}$,
and differentiation gives
$\lambda'(s)=\lambda(s)(\lambda(s)-s)$, so $\lambda'(0)=2/\pi$.
The supporting line at zero therefore gives
\[
 \lambda(s)\ge\sqrt{\frac2\pi}+\frac2\pi s.
\]
Multiplying by $\pi/2$ and subtracting $s$ proves
\eqref{eq:hazard-centroid}. Strict convexity gives the equality
condition.
\end{proof}

\Href{Figure}{fig:hazard-centroid} illustrates the centroid of the
Gaussian tail and the supporting line at zero.

\begin{figure}[!htbp]
\centering
\begin{tabular}{@{}c@{\hspace{0.3cm}}c@{}}\begin{tikzpicture}[x=1cm,y=1cm,>=Stealth,font=\footnotesize,line cap=round,line join=round]
\path[use as bounding box] (-3.0,-1.32) rectangle (3.85,3.85);
\node[] at (0.400000,3.620000) {(a) The centroid of a Gaussian tail};
\fill[blue!14] (0.721000,0.000000) -- (0.721000,2.029651) -- (0.750131,1.989067) -- (0.779263,1.947736) -- (0.808394,1.905739) -- (0.837525,1.863157) -- (0.866657,1.820070) -- (0.895788,1.776557) -- (0.924919,1.732698) -- (0.954051,1.688571) -- (0.983182,1.644251) -- (1.012313,1.599815) -- (1.041444,1.555335) -- (1.070576,1.510882) -- (1.099707,1.466527) -- (1.128838,1.422335) -- (1.157970,1.378372) -- (1.187101,1.334700) -- (1.216232,1.291378) -- (1.245364,1.248463) -- (1.274495,1.206009) -- (1.303626,1.164068) -- (1.332758,1.122686) -- (1.361889,1.081910) -- (1.391020,1.041781) -- (1.420152,1.002338) -- (1.449283,0.963618) -- (1.478414,0.925653) -- (1.507545,0.888472) -- (1.536677,0.852103) -- (1.565808,0.816569) -- (1.594939,0.781892) -- (1.624071,0.748088) -- (1.653202,0.715174) -- (1.682333,0.683161) -- (1.711465,0.652059) -- (1.740596,0.621875) -- (1.769727,0.592615) -- (1.798859,0.564279) -- (1.827990,0.536869) -- (1.857121,0.510382) -- (1.886253,0.484814) -- (1.915384,0.460158) -- (1.944515,0.436407) -- (1.973646,0.413551) -- (2.002778,0.391579) -- (2.031909,0.370477) -- (2.061040,0.350233) -- (2.090172,0.330830) -- (2.119303,0.312252) -- (2.148434,0.294481) -- (2.177566,0.277500) -- (2.206697,0.261289) -- (2.235828,0.245829) -- (2.264960,0.231098) -- (2.294091,0.217076) -- (2.323222,0.203742) -- (2.352354,0.191074) -- (2.381485,0.179050) -- (2.410616,0.167649) -- (2.439747,0.156848) -- (2.468879,0.146626) -- (2.498010,0.136961) -- (2.527141,0.127830) -- (2.556273,0.119213) -- (2.585404,0.111087) -- (2.614535,0.103433) -- (2.643667,0.096229) -- (2.672798,0.089455) -- (2.701929,0.083092) -- (2.731061,0.077119) -- (2.760192,0.071519) -- (2.789323,0.066272) -- (2.818455,0.061361) -- (2.847586,0.056769) -- (2.876717,0.052478) -- (2.905848,0.048473) -- (2.934980,0.044737) -- (2.964111,0.041257) -- (2.993242,0.038017) -- (3.022374,0.035003) -- (3.051505,0.032202) -- (3.080636,0.029602) -- (3.109768,0.027190) -- (3.138899,0.024955) -- (3.168030,0.022885) -- (3.197162,0.020970) -- (3.226293,0.019200) -- (3.255424,0.017565) -- (3.284556,0.016056) -- (3.313687,0.014666) -- (3.342818,0.013385) -- (3.371949,0.012206) -- (3.401081,0.011122) -- (3.430212,0.010127) -- (3.459343,0.009213) -- (3.488475,0.008374) -- (3.517606,0.007606) -- (3.546737,0.006903) -- (3.575869,0.006260) -- (3.605000,0.005672) -- (3.605000,0.000000) -- cycle;
\draw[black!55,thick] (-2.884000,0.051450) -- (-2.842865,0.057492) -- (-2.801730,0.064141) -- (-2.760594,0.071444) -- (-2.719459,0.079452) -- (-2.678324,0.088217) -- (-2.637189,0.097793) -- (-2.596053,0.108235) -- (-2.554918,0.119602) -- (-2.513783,0.131952) -- (-2.472648,0.145345) -- (-2.431513,0.159842) -- (-2.390377,0.175505) -- (-2.349242,0.192396) -- (-2.308107,0.210576) -- (-2.266972,0.230107) -- (-2.225836,0.251048) -- (-2.184701,0.273459) -- (-2.143566,0.297396) -- (-2.102431,0.322912) -- (-2.061296,0.350059) -- (-2.020160,0.378884) -- (-1.979025,0.409428) -- (-1.937890,0.441730) -- (-1.896755,0.475820) -- (-1.855619,0.511725) -- (-1.814484,0.549462) -- (-1.773349,0.589042) -- (-1.732214,0.630466) -- (-1.691079,0.673728) -- (-1.649943,0.718811) -- (-1.608808,0.765689) -- (-1.567673,0.814324) -- (-1.526538,0.864668) -- (-1.485403,0.916661) -- (-1.444267,0.970232) -- (-1.403132,1.025296) -- (-1.361997,1.081760) -- (-1.320862,1.139514) -- (-1.279726,1.198438) -- (-1.238591,1.258401) -- (-1.197456,1.319258) -- (-1.156321,1.380853) -- (-1.115186,1.443022) -- (-1.074050,1.505585) -- (-1.032915,1.568358) -- (-0.991780,1.631145) -- (-0.950645,1.693741) -- (-0.909509,1.755937) -- (-0.868374,1.817515) -- (-0.827239,1.878255) -- (-0.786104,1.937931) -- (-0.744969,1.996316) -- (-0.703833,2.053183) -- (-0.662698,2.108305) -- (-0.621563,2.161456) -- (-0.580428,2.212416) -- (-0.539292,2.260968) -- (-0.498157,2.306903) -- (-0.457022,2.350021) -- (-0.415887,2.390128) -- (-0.374752,2.427047) -- (-0.333616,2.460607) -- (-0.292481,2.490657) -- (-0.251346,2.517055) -- (-0.210211,2.539679) -- (-0.169075,2.558423) -- (-0.127940,2.573197) -- (-0.086805,2.583932) -- (-0.045670,2.590577) -- (-0.004535,2.593100) -- (0.036601,2.591488) -- (0.077736,2.585750) -- (0.118871,2.575913) -- (0.160006,2.562024) -- (0.201142,2.544148) -- (0.242277,2.522371) -- (0.283412,2.496795) -- (0.324547,2.467539) -- (0.365682,2.434740) -- (0.406818,2.398548) -- (0.447953,2.359128) -- (0.489088,2.316658) -- (0.530223,2.271328) -- (0.571358,2.223335) -- (0.612494,2.172887) -- (0.653629,2.120200) -- (0.694764,2.065494) -- (0.735899,2.008992) -- (0.777035,1.950922) -- (0.818170,1.891511) -- (0.859305,1.830986) -- (0.900440,1.769574) -- (0.941575,1.707495) -- (0.982711,1.644969) -- (1.023846,1.582207) -- (1.064981,1.519414) -- (1.106116,1.456788) -- (1.147252,1.394517) -- (1.188387,1.332780) -- (1.229522,1.271746) -- (1.270657,1.211574) -- (1.311792,1.152409) -- (1.352928,1.094386) -- (1.394063,1.037629) -- (1.435198,0.982247) -- (1.476333,0.928339) -- (1.517469,0.875991) -- (1.558604,0.825278) -- (1.599739,0.776262) -- (1.640874,0.728993) -- (1.682009,0.683512) -- (1.723145,0.639846) -- (1.764280,0.598016) -- (1.805415,0.558030) -- (1.846550,0.519887) -- (1.887686,0.483580) -- (1.928821,0.449091) -- (1.969956,0.416398) -- (2.011091,0.385469) -- (2.052226,0.356269) -- (2.093362,0.328756) -- (2.134497,0.302884) -- (2.175632,0.278603) -- (2.216767,0.255861) -- (2.257903,0.234600) -- (2.299038,0.214764) -- (2.340173,0.196291) -- (2.381308,0.179121) -- (2.422443,0.163193) -- (2.463579,0.148444) -- (2.504714,0.134813) -- (2.545849,0.122238) -- (2.586984,0.110660) -- (2.628119,0.100019) -- (2.669255,0.090257) -- (2.710390,0.081318) -- (2.751525,0.073147) -- (2.792660,0.065693) -- (2.833796,0.058904) -- (2.874931,0.052733) -- (2.916066,0.047133) -- (2.957201,0.042060) -- (2.998336,0.037474) -- (3.039472,0.033334) -- (3.080607,0.029605) -- (3.121742,0.026251) -- (3.162877,0.023239) -- (3.204013,0.020541) -- (3.245148,0.018127) -- (3.286283,0.015971) -- (3.327418,0.014049) -- (3.368553,0.012339) -- (3.409689,0.010819) -- (3.450824,0.009472) -- (3.491959,0.008279) -- (3.533094,0.007225) -- (3.574230,0.006295) -- (3.615365,0.005476) -- (3.656500,0.004756);
\draw[blue!65!black,thick] (0.721000,2.029651) -- (0.750131,1.989067) -- (0.779263,1.947736) -- (0.808394,1.905739) -- (0.837525,1.863157) -- (0.866657,1.820070) -- (0.895788,1.776557) -- (0.924919,1.732698) -- (0.954051,1.688571) -- (0.983182,1.644251) -- (1.012313,1.599815) -- (1.041444,1.555335) -- (1.070576,1.510882) -- (1.099707,1.466527) -- (1.128838,1.422335) -- (1.157970,1.378372) -- (1.187101,1.334700) -- (1.216232,1.291378) -- (1.245364,1.248463) -- (1.274495,1.206009) -- (1.303626,1.164068) -- (1.332758,1.122686) -- (1.361889,1.081910) -- (1.391020,1.041781) -- (1.420152,1.002338) -- (1.449283,0.963618) -- (1.478414,0.925653) -- (1.507545,0.888472) -- (1.536677,0.852103) -- (1.565808,0.816569) -- (1.594939,0.781892) -- (1.624071,0.748088) -- (1.653202,0.715174) -- (1.682333,0.683161) -- (1.711465,0.652059) -- (1.740596,0.621875) -- (1.769727,0.592615) -- (1.798859,0.564279) -- (1.827990,0.536869) -- (1.857121,0.510382) -- (1.886253,0.484814) -- (1.915384,0.460158) -- (1.944515,0.436407) -- (1.973646,0.413551) -- (2.002778,0.391579) -- (2.031909,0.370477) -- (2.061040,0.350233) -- (2.090172,0.330830) -- (2.119303,0.312252) -- (2.148434,0.294481) -- (2.177566,0.277500) -- (2.206697,0.261289) -- (2.235828,0.245829) -- (2.264960,0.231098) -- (2.294091,0.217076) -- (2.323222,0.203742) -- (2.352354,0.191074) -- (2.381485,0.179050) -- (2.410616,0.167649) -- (2.439747,0.156848) -- (2.468879,0.146626) -- (2.498010,0.136961) -- (2.527141,0.127830) -- (2.556273,0.119213) -- (2.585404,0.111087) -- (2.614535,0.103433) -- (2.643667,0.096229) -- (2.672798,0.089455) -- (2.701929,0.083092) -- (2.731061,0.077119) -- (2.760192,0.071519) -- (2.789323,0.066272) -- (2.818455,0.061361) -- (2.847586,0.056769) -- (2.876717,0.052478) -- (2.905848,0.048473) -- (2.934980,0.044737) -- (2.964111,0.041257) -- (2.993242,0.038017) -- (3.022374,0.035003) -- (3.051505,0.032202) -- (3.080636,0.029602) -- (3.109768,0.027190) -- (3.138899,0.024955) -- (3.168030,0.022885) -- (3.197162,0.020970) -- (3.226293,0.019200) -- (3.255424,0.017565) -- (3.284556,0.016056) -- (3.313687,0.014666) -- (3.342818,0.013385) -- (3.371949,0.012206) -- (3.401081,0.011122) -- (3.430212,0.010127) -- (3.459343,0.009213) -- (3.488475,0.008374) -- (3.517606,0.007606) -- (3.546737,0.006903) -- (3.575869,0.006260) -- (3.605000,0.005672);
\draw[->] (-2.884000,0.000000) -- (3.811000,0.000000);
\draw[black!65,dashed] (0.721000,0.000000) -- (0.721000,2.029651);
\draw[orange!85!black,thick] (1.329214,0.000000) -- (1.329214,2.177500);
\node[text=orange!85!black] at (1.329214,2.372500) {$\lambda(s)$};
\node[] at (0.000000,-0.227500) {$0$};
\node[] at (0.721000,-0.227500) {$s$};
\node[] at (3.811000,-0.227500) {$t$};
\node[] at (-0.750000,2.950000) {$e^{-t^2/2}$};
\draw[blue!65!black,<->] (0.721000,-0.585000) -- (1.329214,-0.585000);
\node[text=blue!65!black] at (1.025107,-1.007500) {$\lambda(s)-s$};
\fill[orange!85!black] (1.329214,0.000000) circle (1.5pt);
\end{tikzpicture} & \begin{tikzpicture}[x=1cm,y=1cm,>=Stealth,font=\footnotesize,line cap=round,line join=round]
\path[use as bounding box] (-3.0,-1.32) rectangle (3.85,3.85);
\node[] at (0.400000,3.620000) {(b) Convexity and the tangent at zero};
\draw[->] (-2.678000,0.000000) -- (3.399000,0.000000);
\draw[->] (0.000000,-0.855000) -- (0.000000,2.897500);
\draw[blue!65!black,thick] (-2.626500,0.014756) -- (-2.590223,0.016142) -- (-2.553947,0.017637) -- (-2.517670,0.019248) -- (-2.481393,0.020981) -- (-2.445116,0.022843) -- (-2.408840,0.024842) -- (-2.372563,0.026984) -- (-2.336286,0.029277) -- (-2.300009,0.031728) -- (-2.263733,0.034345) -- (-2.227456,0.037135) -- (-2.191179,0.040106) -- (-2.154903,0.043267) -- (-2.118626,0.046624) -- (-2.082349,0.050186) -- (-2.046072,0.053960) -- (-2.009796,0.057955) -- (-1.973519,0.062178) -- (-1.937242,0.066636) -- (-1.900965,0.071338) -- (-1.864689,0.076290) -- (-1.828412,0.081500) -- (-1.792135,0.086976) -- (-1.755858,0.092723) -- (-1.719582,0.098750) -- (-1.683305,0.105061) -- (-1.647028,0.111665) -- (-1.610752,0.118566) -- (-1.574475,0.125770) -- (-1.538198,0.133283) -- (-1.501921,0.141109) -- (-1.465645,0.149255) -- (-1.429368,0.157723) -- (-1.393091,0.166518) -- (-1.356814,0.175644) -- (-1.320538,0.185105) -- (-1.284261,0.194903) -- (-1.247984,0.205041) -- (-1.211708,0.215521) -- (-1.175431,0.226345) -- (-1.139154,0.237515) -- (-1.102877,0.249032) -- (-1.066601,0.260896) -- (-1.030324,0.273109) -- (-0.994047,0.285671) -- (-0.957770,0.298581) -- (-0.921494,0.311838) -- (-0.885217,0.325443) -- (-0.848940,0.339393) -- (-0.812664,0.353687) -- (-0.776387,0.368324) -- (-0.740110,0.383302) -- (-0.703833,0.398618) -- (-0.667557,0.414270) -- (-0.631280,0.430255) -- (-0.595003,0.446571) -- (-0.558726,0.463213) -- (-0.522450,0.480179) -- (-0.486173,0.497465) -- (-0.449896,0.515067) -- (-0.413619,0.532982) -- (-0.377343,0.551205) -- (-0.341066,0.569733) -- (-0.304789,0.588560) -- (-0.268513,0.607684) -- (-0.232236,0.627099) -- (-0.195959,0.646801) -- (-0.159682,0.666786) -- (-0.123406,0.687048) -- (-0.087129,0.707583) -- (-0.050852,0.728386) -- (-0.014575,0.749453) -- (0.021701,0.770779) -- (0.057978,0.792358) -- (0.094255,0.814187) -- (0.130531,0.836261) -- (0.166808,0.858575) -- (0.203085,0.881123) -- (0.239362,0.903902) -- (0.275638,0.926906) -- (0.311915,0.950132) -- (0.348192,0.973574) -- (0.384469,0.997228) -- (0.420745,1.021089) -- (0.457022,1.045154) -- (0.493299,1.069416) -- (0.529575,1.093873) -- (0.565852,1.118520) -- (0.602129,1.143353) -- (0.638406,1.168367) -- (0.674682,1.193559) -- (0.710959,1.218924) -- (0.747236,1.244458) -- (0.783513,1.270158) -- (0.819789,1.296020) -- (0.856066,1.322040) -- (0.892343,1.348214) -- (0.928619,1.374539) -- (0.964896,1.401012) -- (1.001173,1.427628) -- (1.037450,1.454384) -- (1.073726,1.481278) -- (1.110003,1.508306) -- (1.146280,1.535464) -- (1.182557,1.562750) -- (1.218833,1.590161) -- (1.255110,1.617693) -- (1.291387,1.645344) -- (1.327664,1.673111) -- (1.363940,1.700991) -- (1.400217,1.728982) -- (1.436494,1.757080) -- (1.472770,1.785284) -- (1.509047,1.813590) -- (1.545324,1.841997) -- (1.581601,1.870502) -- (1.617877,1.899103) -- (1.654154,1.927797) -- (1.690431,1.956582) -- (1.726708,1.985456) -- (1.762984,2.014417) -- (1.799261,2.043463) -- (1.835538,2.072592) -- (1.871814,2.101801) -- (1.908091,2.131090) -- (1.944368,2.160456) -- (1.980645,2.189897) -- (2.016921,2.219412) -- (2.053198,2.248999) -- (2.089475,2.278656) -- (2.125752,2.308382) -- (2.162028,2.338175) -- (2.198305,2.368033) -- (2.234582,2.397956) -- (2.270858,2.427941) -- (2.307135,2.457987) -- (2.343412,2.488093) -- (2.379689,2.518257) -- (2.415965,2.548478) -- (2.452242,2.578755) -- (2.488519,2.609087) -- (2.524796,2.639471) -- (2.561072,2.669908) -- (2.597349,2.700396) -- (2.633626,2.730934) -- (2.669903,2.761521) -- (2.706179,2.792155) -- (2.742456,2.822836) -- (2.778733,2.853562) -- (2.815009,2.884333) -- (2.851286,2.915148) -- (2.887563,2.946005) -- (2.923840,2.976905) -- (2.960116,3.007845) -- (2.996393,3.038825) -- (3.032670,3.069844) -- (3.068947,3.100902) -- (3.105223,3.131997) -- (3.141500,3.163129);
\draw[black!65,dashed] (-2.575000,-0.753982) -- (3.141500,2.602596);
\fill[blue!65!black] (0.000000,0.757990) circle (1.5pt);
\node[anchor=east] at (-0.133900,0.871990) {$\sqrt{2/\pi}$};
\node[text=blue!65!black] at (2.575000,2.878500) {$\lambda(s)$};
\node[fill=white,inner sep=1pt] at (1.390500,0.551000) {$\sqrt{2/\pi}+(2/\pi)s$};
\node[] at (3.347500,-0.171000) {$s$};
\node[] at (-0.144200,-0.152000) {$0$};
\draw[thin] (-2.060000,-0.042750) -- (-2.060000,0.042750);
\draw[thin] (2.060000,-0.042750) -- (2.060000,0.042750);
\node[] at (-2.060000,-0.152000) {$-2$};
\node[] at (2.060000,-0.152000) {$2$};
\end{tikzpicture}\end{tabular}
\caption{The one-dimensional inequality. The left panel shows the
centroid $\lambda(s)$ of the Gaussian tail $(s,\infty)$ and its
distance $\lambda(s)-s$ from the endpoint (here $s=0.7$). The right
panel shows the supporting line to $\lambda$ at zero, which gives
\eqref{eq:hazard-centroid}.}
\label{fig:hazard-centroid}
\end{figure}

To express $\ell_i$ and $h_i$ in terms of $\lambda$, we return to
the slice formula \eqref{eq:coordinate-slice}. It says that, on each
coordinate ray, the Gaussian density becomes a one-dimensional
Gaussian truncated at the initial point of the ray. We make this
identification explicit before averaging.

Fix $i\in[d]$, write $u=u_i$, and use $t=x[i]$ and
$y=(x[j])_{j\ne i}$ as above.
For fixed $y\in\R_+^{d-1}$, set
\begin{equation}\label{eq:ray-parameters}
 \tau=\sqrt{G_{ii}}\,t,\qquad
 s=\frac{\sum_{j\ne i}G_{ij}y[j]}{\sqrt{G_{ii}}}.
\end{equation}
The ray has image $B(t,y)=B(0,y)+\tau u$, and its coordinate
along $u$ is $s+\tau$, since $\iprod{B(0,y)}{u}=s$.
Completing the square in \eqref{eq:coordinate-slice} gives
\[
 f_G(t,y)=f_G(0,y)e^{s^2/2}e^{-(s+\tau)^2/2}.
\]

\begin{lem}[The ray-average formulas]\label{lem:ray-averages}
For each $i\in[d]$, with $s=s(y)$ as in
\eqref{eq:ray-parameters},
\begin{equation}\label{eq:length-ray-average}
 \ell_i=\frac{\int_{\R_+^{d-1}}(\lambda(s)-s)\parenth{\int_0^\infty f_G(t,y)\di t}\di y}{\int_{\R_+^d}f_G},
\end{equation}
and
\begin{equation}\label{eq:projection-ray-average}
 h_i=\frac{\int_{\R_+^{d-1}}\lambda(s)\parenth{\int_0^\infty f_G(t,y)\di t}\di y}{\int_{\R_+^d}f_G}.
\end{equation}
\end{lem}

\begin{proof}
Fix $i$ and $y$. The completed-square formula above shows that,
after normalization on the ray, $s+\tau$ has the standard Gaussian
distribution conditioned to lie in $(s,\infty)$.
Its mean is $\lambda(s)$, so the mean of the distance $\tau$
from the initial point is $\lambda(s)-s$. In integral form,
\begin{equation}\label{eq:ray-length-average}
 \int_0^\infty\sqrt{G_{ii}}\,t f_G(t,y)\di t
 =(\lambda(s)-s)\int_0^\infty f_G(t,y)\di t,
\end{equation}
and
\begin{equation}\label{eq:ray-projection-average}
 \int_0^\infty(s+\sqrt{G_{ii}}\,t)f_G(t,y)\di t
 =\lambda(s)\int_0^\infty f_G(t,y)\di t.
\end{equation}
Indeed, the change of variables $\tau=\sqrt{G_{ii}}t$ contributes
the common factor $f_G(0,y)e^{s^2/2}/\sqrt{G_{ii}}$.
Dividing by the ray mass cancels this factor, and the change of
variables from $\tau$ to $s+\tau$ gives precisely the defining
integrals for the tail centroid.

Integrating these identities over $y$ expresses $\ell_i$ and $h_i$
as averages with the same weights, namely the integrals of $f_G$
along the coordinate rays. For $\ell_i$, the left-hand side becomes
$\sqrt{G_{ii}}\mu[i]\int_{\R_+^d}f_G$. For $h_i$, it becomes
$\iprod{c}{u}\int_{\R_+^d}f_G$, because
$\iprod{B(t,y)}{u}=s+\sqrt{G_{ii}}t$.
Dividing by $\int_{\R_+^d}f_G$ gives the two formulas.
\end{proof}

The desired lower bound in \eqref{eq:coordinate-segment} is now
exactly the average of \eqref{eq:hazard-centroid} with these common
weights.

\begin{lem}[The coordinate linear estimate]\label{lem:centroid-line}
For each $i\in[d]$,
\begin{equation}\label{eq:centroid-line}
 \ell_i+\parenth{\frac\pi2-1}h_i\ge\sqrt{\frac\pi2}.
\end{equation}
Equivalently, in terms of $G$ and $\mu$,
\begin{equation}\label{eq:centroid-line-matrix}
 \sqrt{G_{ii}}\,\mu[i]+\parenth{\frac\pi2-1}\frac{(G\mu)[i]}{\sqrt{G_{ii}}}\ge\sqrt{\frac\pi2}.
\end{equation}
Equality holds if and only if $G_{ij}=0$ for every $j\ne i$.
\end{lem}

\begin{proof}
Fix $i\in[d]$. For each $y\in\R_+^{d-1}$, consider the coordinate
ray $\{(t,y):t>0\}$, where $t$ occupies position $i$.
Completing the square in the slice formula
\eqref{eq:coordinate-slice} gives the truncation parameter
$s=s(y)=\sum_{j\ne i}G_{ij}y[j]/\sqrt{G_{ii}}$, as in
\eqref{eq:ray-parameters}.
Apply \eqref{eq:hazard-centroid} to this $s$, multiply by
$\int_0^\infty f_G(t,y)\di t$, and integrate over $y$.
The resulting integral inequality is
\[
 \int_{\R_+^{d-1}}\parenth{\lambda(s)-s+\parenth{\frac\pi2-1}\lambda(s)}\parenth{\int_0^\infty f_G(t,y)\di t}\di y
 \ge\sqrt{\frac\pi2}\int_{\R_+^d}f_G.
\]
Here the right-hand side follows from Fubini's theorem.
After division by $\int_{\R_+^d}f_G$,
\Href{Lemma}{lem:ray-averages} identifies the two averages on
the left with $\ell_i$ and $h_i$, proving
\eqref{eq:centroid-line}.

All ray masses are positive. Equality therefore holds if and only
if $s=0$ for almost every $y\in\R_+^{d-1}$. The linear function
$\sum_{j\ne i}G_{ij}y[j]$ must then vanish identically, which is
equivalent to $G_{ij}=0$ for every $j\ne i$.
\end{proof}

To obtain the upper endpoint in \eqref{eq:coordinate-segment}, we
use the boundary identity and log-concavity of a coordinate tail.

\begin{lem}\label{lem:centroid-product}
For each $i\in[d]$,
\begin{equation}\label{eq:centroid-product}
 \ell_i h_i=\mu[i](G\mu)[i]\le1.
\end{equation}
\end{lem}

\begin{proof}
Let $S_i(t)$ be the tail probability of coordinate $i$ under the
Gaussian distribution conditioned on the orthant:
\[
 S_i(t)=\frac{\int_{\R_+^{d-1}}\int_t^\infty f_G(\tau,y)\di\tau\di y}{\int_{\R_+^d}f_G},\qquad t\ge0.
\]
Its numerator is $\int_{\R_+^d}f_G(x+te)\di x$, where $e$ is
the $i$th coordinate vector. The integrand, extended by zero
outside the orthant in the $x$ variable, is log-concave jointly
in $(x,t)$. Pr\'ekopa's theorem \cite{prekopa1971logarithmic}
therefore shows that $S_i$ is log-concave.
Moreover, $S_i(0)=1$, and \Href{Lemma}{lem:boundary} gives
\[
 S_i'(0)=-\frac{\int_{\R_+^{d-1}}f_G(0,y)\di y}{\int_{\R_+^d}f_G}=-(G\mu)[i].
\]
The tangent line bound for the concave function $\log S_i$ implies
$S_i(t)\le e^{-t(G\mu)[i]}$. Integrating the tail probability,
we obtain
\[
 \mu[i]=\int_0^\infty S_i(t)\di t
 \le\int_0^\infty e^{-t(G\mu)[i]}\di t=\frac1{(G\mu)[i]}.
\]
Multiplication by $(G\mu)[i]>0$ proves the claim.
\end{proof}

\FloatBarrier
\subsection{Completion of the proofs}

\begin{proof}[Proof of \Href{Lemma}{lem:coordinate-quadratic}]
Fix $i\in[d]$ and write $\ell=\ell_i$ and $h=h_i>0$.
The preceding two coordinate bounds give the interval
\eqref{eq:coordinate-segment} for $\ell$.
The left-hand side of \eqref{eq:coordinate-quadratic} is affine
in $\ell$, so it is enough to check its values at the endpoints.
At the left endpoint, direct expansion gives
\[
 \frac\pi2 h^2+\parenth{4\sqrt{\frac2\pi}-2h}
 \parenth{\sqrt{\frac\pi2}-\parenth{\frac\pi2-1}h}
 =1+\frac4\pi+\parenth{\frac{3\pi}2-2}\parenth{h-\sqrt{\frac2\pi}}^2.
\]
At the right endpoint, using the arithmetic mean--geometric mean
inequality for the three positive numbers
$\frac\pi2 h^2$, $\frac{2\sqrt{2/\pi}}h$, and
$\frac{2\sqrt{2/\pi}}h$, whose product is $4$, we get
\[
 \frac\pi2 h^2+\frac{4\sqrt{2/\pi}}h-2
 \ge3\sqrt[3]{4}-2>1+\frac4\pi.
\]
Thus, the desired bound holds at both endpoints and hence throughout
the interval.

The bound at the right endpoint is strict. Equality therefore
requires $\ell$ to be the left endpoint and $h=\sqrt{2/\pi}$, which gives
$\ell=\sqrt{2/\pi}$. These values give equality in
\Href{Lemma}{lem:centroid-line}, so $G_{ij}=0$ for every $j\ne i$.
Conversely, under this condition the $i$th coordinate is independent
of the others, and direct integration gives
$\ell=h=\sqrt{2/\pi}$.
\end{proof}

\begin{proof}[Proof of \Href{Theorem}{thm:centroid-estimate}]
For diagonal $G$, direct integration gives
$\mu[i]=\sqrt{2/(\pi G_{ii})}$. Hence, $h_i=\sqrt{2/\pi}$ and
$\iprod{G^{-1}\delta_G}{\delta_G}=d$, proving equality in
\eqref{eq:centroid-quadratic}. This also settles $d=1$.

For $d\ge2$, sum \eqref{eq:coordinate-quadratic} over $i$ and
use the identity \eqref{eq:centroid-square}. Since the squared norm
is nonnegative, we obtain
\[
 \frac\pi2\sum_{i\in[d]}h_i^2+\frac4\pi\enorm{v}^2
 \ge\sum_{i\in[d]}\parenth{\frac\pi2 h_i^2+\parenth{4\sqrt{\frac2\pi}-2h_i}\ell_i}
 \ge\parenth{1+\frac4\pi}d.
\]
Together with \eqref{eq:coordinate-projection} and
\eqref{eq:tangent-norm}, this is \eqref{eq:centroid-quadratic}.
Equality forces equality in every coordinate estimate, so
\Href{Lemma}{lem:coordinate-quadratic} implies that $G$ is diagonal.
\end{proof}

\begin{rem}\label{rem:centroid-linear}
The same endpoint argument also gives the linear centroid inequality
\begin{equation}\label{eq:centroid-linear}
 \sqrt{\frac\pi2}\sum_{i\in[d]}\frac{(G\mu)[i]}{\sqrt{G_{ii}}}
 +\frac2\pi\iprod{G^{-1}\delta_G}{\delta_G}
 \ge\parenth{1+\frac2\pi}d,
\end{equation}
with equality if and only if $G$ is diagonal.
Indeed, direct substitution at the two endpoints of
\eqref{eq:coordinate-segment} gives
\[
 \sqrt{\frac\pi2}\,h_i+\parenth{2\sqrt{\frac2\pi}-h_i}\ell_i
 \ge1+\frac2\pi,
\]
with equality precisely when $\ell_i=h_i=\sqrt{2/\pi}$.
Summing over $i$ and adding the nonnegative square
$\enorm{c-\sqrt{2/\pi}\,v}^2$ gives
\eqref{eq:centroid-linear} by
\eqref{eq:centroid-geometric-identities} and
\eqref{eq:tangent-norm}. The equality condition follows as in
the proof of \Href{Theorem}{thm:centroid-estimate}.
\end{rem}

\begin{proof}[Proof of \Href{Theorem}{thm:gaussian}]
Combining \Href{Corollary}{cor:projection-centroid} with
\Href{Theorem}{thm:centroid-estimate} gives \eqref{eq:gaussian}.
Equality forces equality in \Href{Theorem}{thm:centroid-estimate},
so $G$ is diagonal. Conversely, for diagonal $G$, the orthant
formula gives $\Psi_G=2d$, and
$\iprod{G^{-1}\delta_G}{\delta_G}=d$, so equality holds.
This also covers $d=1$.
\end{proof}

\FloatBarrier

\section{A quadratic extension}\label{sec:quadratic}

By polarity, \Href{Theorem}{thm:halfspaces} gives a lower bound on the
common norm of $2d$ vectors whose convex hull contains the unit ball.
The first conjecture below asks for a bound on the sum of their
squared norms when the norms may differ. The second proposes a bound
in terms of the Gaussian vertex measure of the polar of a tight
frame and implies the first conjecture.

\begin{dfn}\label{dfn:tight-frame}
A set of vectors $v_1, \dots, v_m \in \R^d$ is a \emph{tight frame} if
\[
 \sum_{i \in [m]} v_i \otimes v_i = \id_d.
\]
\end{dfn}

For such a frame, write
\[
 C = \{x \in \R^d: \iprod{v_i}{x} \leq 1, \ i \in [m]\}.
\]
This is the polar of $\cof{v_i: i \in [m]}$.

\subsection{Open problems}

\begin{conj}\label{conj:energy}
Let $b_1, \dots, b_{2d} \in \R^d$ satisfy
$\ball{d} \subseteq \cof{b_i: i \in [2d]}$. Then
\begin{equation} \label{eq:energy}
 \sum_{i \in [2d]} \enorm{b_i}^2 \geq 2d^2.
\end{equation}
\end{conj}

\begin{conj}\label{conj:gaussian-functional}
Let $v_1, \dots, v_{2d} \in \R^d$ be a tight frame with
$0 \in \operatorname{int} \cof{v_i: i \in [2d]}$, and put
$C = \cof{v_i: i \in [2d]}^\circ$.
Then
\[
 \tr \sqrt{A_\gamma(C)} \geq \sqrt{2}\, d.
\]
\end{conj}

Both bounds are attained by suitably scaled cross-polytopes. Take
$b_i \in \{\pm \sqrt{d}\, e_j: j \in [d]\}$ in
\Href{Conjecture}{conj:energy} and
$v_i \in \{\pm e_j/\sqrt{2}: j \in [d]\}$ in
\Href{Conjecture}{conj:gaussian-functional}. In the latter case,
$C = \sqrt{2}[-1, 1]^d$ and $A_\gamma(C) = 2 \id_d$.
To see the equal-norm case of \Href{Conjecture}{conj:energy}, let $r$
be the common norm of the $b_i$. The polar of
$\cof{b_i/r: i \in [2d]}$ is contained in $r \ball{d}$, so
\Href{Theorem}{thm:halfspaces} gives $r \geq \sqrt{d}$.

\Href{Conjecture}{conj:energy} is also related to the vertex index.
Bezdek and Litvak~\cite{bezdek2007vertex} defined the \emph{vertex index}
of an origin-symmetric convex body $K$ by
\[
 \operatorname{vein}(K)
 = \inf \braces{\sum_{i \in [N]} \norm{b_i}_K:
  K \subseteq \cof{b_i: i \in [N]}}.
\]
Here $\norm{\cdot}_K$ is the norm with unit ball $K$, and the infimum
is over all finite sets. They proposed the following conjecture.

\begin{conj}[Bezdek--Litvak, {\cite[Conjecture~B]{bezdek2007vertex}}]
\label{conj:vertex-index}
For every $d \geq 2$,
\[
 \operatorname{vein}(\ball{d}) = 2d^{3/2}.
\]
\end{conj}

This implies \Href{Conjecture}{conj:energy}, since the Cauchy--Schwarz
inequality gives
\[
 \sum_{i \in [2d]} \enorm{b_i}^2
 \geq \frac{1}{2d} \parenth{\sum_{i \in [2d]} \enorm{b_i}}^2 \geq 2d^2.
\]
The result of Bezdek and Litvak in dimensions two and three
\cite[Theorem~4.1]{bezdek2007vertex} gives
\Href{Conjecture}{conj:energy} in those dimensions.

\subsection{Quadratic forms and measures on vertices}

For a bounded full-dimensional polytope $C$ containing the origin in
its interior, a positive definite operator $Q$, and a probability
measure $\nu$ on $\vertices(C)$, put
\[
 S(C, Q) = \max_{x \in C} \iprod{Qx}{x}, \qquad
 A_\nu = \sum_{v \in \vertices(C)} \nu(v)\, v \otimes v.
\]
For polars of tight frames of $2d$ vectors, \Href{Conjecture}{conj:energy}
is equivalent to
\begin{equation} \label{eq:quadratic}
 S(C, Q) \tr Q^{-1} \geq 2d^2 \qquad \text{for every } Q \succ 0.
\end{equation}
Indeed, the vectors $b_i = \sqrt{S(C, Q)}\, Q^{-1/2} v_i$ have a convex
hull containing $\ball{d}$, and their total squared norm is
$S(C, Q) \tr Q^{-1}$. Conversely, given the vectors in
\Href{Conjecture}{conj:energy}, put
\[
 M = \sum_{i \in [2d]} b_i \otimes b_i, \qquad
 v_i = M^{-1/2} b_i, \qquad Q = M^{-1}.
\]
Then $v_1, \dots, v_{2d}$ is a tight frame and
$C \subseteq M^{1/2} \ball{d}$. Thus, $S(C, Q) \leq 1$, and
\eqref{eq:quadratic} implies \eqref{eq:energy}.

The following identity reformulates \eqref{eq:quadratic} in terms of
probability measures on the vertices of $C$.

\begin{prp}\label{prp:duality}
For every bounded full-dimensional polytope $C$ with
$0 \in \operatorname{int} C$,
\begin{equation} \label{eq:duality}
 \inf_{Q \succ 0} S(C, Q) \tr Q^{-1}
 = \max_\nu \parenth{\tr \sqrt{A_\nu}}^2
 = \min_{\substack{L \succ 0\\C \subseteq L^{1/2} \ball{d}}} \tr L.
\end{equation}
The maximum is over probability measures on $\vertices(C)$.
\end{prp}

\begin{proof}
For every $Q \succ 0$ and every $\nu$, the Cauchy--Schwarz inequality
for the Hilbert--Schmidt inner product gives
\[
 \parenth{\tr \sqrt{A_\nu}}^2
 = \iprod{Q^{1/2} A_\nu^{1/2}}{Q^{-1/2}}^2
 \leq \iprod{Q}{A_\nu} \tr Q^{-1}
 \leq S(C, Q) \tr Q^{-1}.
\]
Choose a measure maximizing $\tr\sqrt{A_\nu}$, write $A_*$ for its
second moment operator, and put
$g_* = \tr \sqrt{A_*}$. The operator $A_*$ is positive definite.
Indeed, the vertices span $\R^d$, so there is a second moment operator
$A_0 \succeq c \id_d$ with $c > 0$. If $A_*$ were singular, the eigenvalues of
$(1 - \varepsilon)A_*+\varepsilon A_0$ would be at least
$(1 - \varepsilon) \lambda_j(A_*) + \varepsilon c$ for $j \in [d]$.
Thus, mixing with $A_0$ gives a gain of order $\sqrt{\varepsilon}$ on the
kernel and a possible loss of order $\varepsilon$ elsewhere, contradicting
maximality.

The derivative of $\tr \sqrt{A}$ at $A_*$ in the direction $H$ is
$\frac{1}{2} \iprod{A_*^{-1/2}}{H}$. Optimality in the directions
$v \otimes v - A_*$ gives
\[
 \iprod{A_*^{-1/2} v}{v} \leq g_*
 \qquad \text{for every } v \in \vertices(C).
\]
For $Q = A_*^{-1/2}$, we therefore have
$S(C, Q) \leq g_*$ and $\tr Q^{-1} = g_*$, proving the first equality.
Finally, $L = S(C, Q)Q^{-1}$ satisfies $C \subseteq L^{1/2} \ball{d}$ and
$\tr L = S(C, Q) \tr Q^{-1}$. Conversely, every admissible $L$ satisfies
$S(C, L^{-1}) \leq 1$. This proves the second equality and attainment.
\end{proof}

By \Href{Proposition}{prp:duality}, \eqref{eq:quadratic} is equivalent
to the existence of a probability measure $\nu$ on the vertices with
$\tr \sqrt{A_\nu} \geq \sqrt{2}\, d$. For the Gaussian vertex
measure, the same Cauchy--Schwarz calculation gives
\begin{equation} \label{eq:gaussian-functional-bound}
 S(C, Q) \tr Q^{-1}
 \geq \iprod{Q}{A_\gamma(C)} \tr Q^{-1}
 \geq \parenth{\tr \sqrt{A_\gamma(C)}}^2.
\end{equation}
This proves the implication from \Href{Conjecture}{conj:gaussian-functional}
to \Href{Conjecture}{conj:energy}.
\Href{Theorem}{thm:gaussian-vertices} controls $\tr A_\gamma(P)$ for
unit facet normals. Passing to a tight frame, however, can change the Gaussian cone
weights appearing in \Href{Conjecture}{conj:gaussian-functional}.

\section{An independent spectral estimate}\label{sec:spectral}

We now bound the quadratic quantity $S(C,Q)\tr Q^{-1}$ from
\eqref{eq:quadratic}. The same estimate gives an independent radius
bound for Zong's problem. Our starting point is
the following result of Ball and Prodromou, which we used with
Nasz\'odi to obtain lower bounds for quantitative Helly theorems
\cite[Section~4]{ivanov2022quantitative}.

\begin{prp}[Ball--Prodromou, {\cite[Theorem~1.4]{Ball2009}}]
\label{prp:ball-prodromou}
Let $v_1, \dots, v_m \in \R^d$ be a tight frame, and let $Q$ be a positive
semidefinite operator on $\R^d$. There is a point $x \in \R^d$ such that
\[
 \abs{\iprod{v_i}{x}} \leq 1 \quad \text{for every } i \in [m],
 \qquad \iprod{Qx}{x} \geq \tr Q.
\]
\end{prp}

The strips in \Href{Proposition}{prp:ball-prodromou} are invariant
under sign changes of the normals. For the one-sided inequalities
$\iprod{v_i}{x} \leq 1$, the directions of the normals also affect the
geometry of $C$. We use this information to improve the estimate.
When the normals sum to zero, the remaining inequalities contribute an
additional term to the bound. Taking Cartesian powers and projecting
onto a hyperplane gives the same estimate in the general case.

The proof below follows the eigenvector flag used by Ball and Prodromou.
A related use of singular subspaces and vertices of sections appears
in the vertex-index estimate of Gluskin and
Litvak~\cite[Section~5]{gluskin2012vertex}.

For integers $m > d$, put
\begin{equation} \label{eq:spectral-coefficients}
 f_m(k) = \frac{mk}{m - k}, \qquad
 c_{m, k} = f_m(k) - f_m(k - 1) = \frac{m^2}{(m - k)(m - k + 1)}
 \qquad \text{for } k \in [d],
\end{equation}
and note that $f_m(0) = 0$.

\begin{thm}\label{thm:spectral}
Let $v_1, \dots, v_m \in \R^d$ be a tight frame with
$0 \in \operatorname{int} \cof{v_i: i \in [m]}$, and put
\[
 C = \{x \in \R^d: \iprod{v_i}{x} \leq 1, \ i \in [m]\}.
\]
Let $Q$ be positive definite, with eigenvalues
$\lambda_1 \geq \dots \geq \lambda_d > 0$. Then
\begin{equation} \label{eq:spectral}
 S(C, Q): = \max_{x \in C} \iprod{Qx}{x}
 \geq \sum_{k \in [d]} c_{m, k} \lambda_k.
\end{equation}
\end{thm}

\begin{proof}
First, assume that $\sum_{i \in [m]} v_i = 0$.
Let $F$ be a linear subspace of dimension $k$, and let $z$ be a vertex of
$C \cap F$. Choose an index set $I \subseteq [m]$ of size $k$ such that the
corresponding inequalities are active at $z$ and their normals projected
onto $F$ are linearly independent. For $j \in [m] \setminus I$, write
$t_j = \iprod{v_j}{z}$. The zero-sum condition gives
$\sum_{j \in [m] \setminus I} t_j = -k$, while the tight-frame identity gives
\begin{equation} \label{eq:face-deficit}
 \enorm{z}^2
 = k + \sum_{j \in [m] \setminus I} t_j^2
 = f_m(k) + \sum_{j \in [m] \setminus I}
  \parenth{t_j + \frac{k}{m - k}}^2.
\end{equation}
In particular, every vertex of this section has squared norm at least
$f_m(k)$.

We apply this bound along a flag of subspaces spanned by eigenvectors.
Choose an orthonormal eigenbasis of $Q$ in decreasing order of
eigenvalue, let $F_k$ be the span of its first $k$ vectors, and put
$\lambda_{d + 1} = 0$. Define
\[
 q_k(x) = \sum_{j \in [k]}(\lambda_j - \lambda_{k + 1})x[j]^2,
 \qquad q_0 = 0.
\]
Starting with $z_0 = 0$, choose a vertex $z_k$ of $C \cap F_k$ maximizing
$q_{k - 1}$. Such a vertex exists because $q_{k - 1}$ is convex and the
section is a polytope. The function $q_{k - 1}$ depends only on the
first $k-1$ coordinates. Since $z_{k - 1}$ also lies in $C\cap F_k$,
we have
\[
 q_k(z_k)
 = q_{k - 1}(z_k) + (\lambda_k - \lambda_{k + 1}) \enorm{z_k}^2 
  \geq 
  q_{k - 1}(z_{k - 1})
  +(\lambda_k - \lambda_{k + 1})f_m(k).
\]
Adding these inequalities yields
\[
 S(C, Q) \geq \sum_{k \in [d]}(\lambda_k - \lambda_{k + 1})f_m(k)
 = \sum_{k \in [d]} c_{m, k} \lambda_k.
\]
This proves the theorem in the zero-sum case.

For the general case, take $N$ orthogonal copies of the tight frame in
$\R^{Nd}$. The polar of their union is $C^N$. Put
$Q_N = Q \oplus \dots \oplus Q$, and choose a hyperplane $E$
perpendicular to the sum of all the copied vectors. If the sum is zero,
choose any hyperplane. Orthogonal projection onto $E$ produces a
tight frame of $Nm$ vectors with sum zero and polar $C^N \cap E$ in $E$.

Let $\pi_E$ denote orthogonal projection onto $E$, and define the
self-adjoint operator $Q_E: E \to E$ by
\[
 Q_E x = \pi_E(Q_Nx) \qquad \text{for } x \in E.
\]
It represents the restriction of the quadratic form to $E$, since
$\iprod{Q_E x}{x} = \iprod{Q_Nx}{x}$ for every $x \in E$.
Write $\mu_1 \geq \dots \geq \mu_{Nd - 1}$ for the eigenvalues of $Q_E$,
and $\Lambda_1 \geq \dots \geq \Lambda_{Nd}$ for those of $Q_N$.
In an orthonormal basis whose last vector is normal to $E$, the
matrix of $Q_E$ is obtained from that of $Q_N$ by deleting the last
row and column. Cauchy's interlacing theorem
\cite[Theorem~4.3.17 and the general form in Theorem~4.3.28]{Horn2012}
gives $\Lambda_k \geq \mu_k \geq \Lambda_{k + 1}$ for
$k \in [Nd - 1]$. For the lower bound, we may therefore discard
the largest entry in the ordered spectrum of $Q_N$. The zero-sum case gives
\[
 N S(C, Q)
  = \max_{x \in C^N} \iprod{Q_Nx}{x}
 \geq \sum_{k \in [Nd - 1]} c_{Nm, k} \mu_k
 \geq \sum_{j \in [d]}(\lambda_j - \lambda_{j + 1})f_{Nm}(Nj - 1).
\]
For the last step, each increment $\lambda_j-\lambda_{j+1}$ contributes
to the first $Nj$ entries of the spectrum of $Q_N$, so it contributes
to the first $Nj-1$ entries after the largest is discarded. Summing
the corresponding coefficients gives $f_{Nm}(Nj-1)$. Divide by $N$ and let
$N \to \infty$. Since
$N^{-1} f_{Nm}(Nj - 1) \to f_m(j)$, this proves \eqref{eq:spectral}.
\end{proof}

\begin{cor}\label{cor:spectral}
Under the assumptions of \Href{Theorem}{thm:spectral},
\begin{equation} \label{eq:quadratic-bound}
 S(C,Q)\tr Q^{-1} \geq m^2\ln^2\frac{m}{m-d}.
\end{equation}
More precisely, $S(C,Q)\tr Q^{-1} \geq \eta_{m,d}^2$, where
\begin{equation} \label{eq:spectral-constant}
 \eta_{m,d} = m\sum_{k\in[d]} \frac{1}{\sqrt{(m-k)(m-k+1)}}.
\end{equation}
\end{cor}

\begin{proof}
The Cauchy--Schwarz inequality and \Href{Theorem}{thm:spectral} give
\[
 S(C, Q) \tr Q^{-1}
 \geq \parenth{\sum_{k \in [d]} c_{m, k} \lambda_k}
  \parenth{\sum_{k \in [d]} \lambda_k^{-1}}
 \geq \parenth{\sum_{k \in [d]} \sqrt{c_{m, k}}}^2 = \eta_{m, d}^2.
\]
For every integer $j \geq 1$,
$\ln((j + 1)/j) \leq 1/\sqrt{j(j + 1)}$. Indeed, the derivative of
$t/\sqrt{1 + t} - \ln(1 + t)$ equals
$(\sqrt{1 + t} - 1)^2/(2(1 + t)^{3/2}) \geq 0$.
Use $t = 1/j$ and sum over $j \in \{m - d, \dots, m - 1\}$ to obtain
$\eta_{m, d} \geq m \ln(m/(m - d))$.
\end{proof}

For a scalar $Q$ and $m=2d$, \Href{Theorem}{thm:spectral} gives the
constant $2d^2$ in \eqref{eq:quadratic}, since
$\sum_{k\in[d]}c_{2d,k}=2d$.

\begin{cor}[A radius bound]\label{cor:spectral-radius}
Let $P$ be the intersection of $m>d$ closed half-spaces in $\R^d$
containing $\ball{d}$. Then $P$ contains a point of norm at least
\begin{equation}\label{eq:spectral-radius}
 \sqrt{m}\,\ln\frac{m}{m-d}.
\end{equation}
More precisely, $P$ contains a point of norm at least
$\eta_{m,d}/\sqrt{m}$.
\end{cor}

\begin{proof}
It suffices to consider bounded intersections.
Move each bounding hyperplane inward until it touches
$\ball{d}$. Let $P_0$ be the resulting intersection and let
$a_1,\dots,a_m$ be the unit outer normals. Since $P_0$ is bounded,
$0\in\operatorname{int}\cof{a_i:i\in[m]}$. Put
\[
 M = \sum_{i\in[m]} a_i \otimes a_i, \qquad
 v_i = M^{-1/2} a_i, \qquad Q = M^{-1}.
\]
The vectors $v_i$ form a tight frame, their polar is $C = M^{1/2} P_0$, and
$S(C,Q)=\max_{x\in P_0}\enorm{x}^2$. Since $\tr Q^{-1}=\tr M=m$,
\Href{Corollary}{cor:spectral} gives both bounds.
\end{proof}

For $m=cd$ with $c>1$, the bound \eqref{eq:spectral-radius} is
$\left(\sqrt{c}\,\ln\frac{c}{c-1}\right)\,\sqrt{d}$. For $c=2$, the logarithmic
estimates give $S(C,Q)\tr Q^{-1}\geq 4(\ln 2)^2d^2$ and the radius
bound $\sqrt{2}\ln 2\,\sqrt{d}>0.98\sqrt{d}$.

Taking the infimum over $Q$ in \Href{Corollary}{cor:spectral} gives a
bound in terms of a probability measure on the vertices.

\begin{cor}[A combinatorial dual inequality]\label{cor:combinatorial-dual}
Let $v_1, \dots, v_m \in \R^d$ be a tight frame with
$0 \in \operatorname{int} \cof{v_i: i \in [m]}$, and let
$C = \{x: \iprod{v_i}{x} \leq 1, \ i \in [m]\}$. There is a probability
measure $\nu$ on $\vertices(C)$ such that
\begin{equation} \label{eq:combinatorial-dual}
 \tr \sqrt{\sum_{v \in \vertices(C)} \nu(v)\, v \otimes v}
 \geq m\ln\frac{m}{m-d}.
\end{equation}
More precisely, the left-hand side is at least $\eta_{m,d}$.
Equivalently, every ellipsoid $L^{1/2}\ball{d}$ centered at the origin
and containing $C$ satisfies $\tr L\geq m^2\ln^2\frac{m}{m-d}$,
with the stronger bound $\tr L\geq \eta_{m,d}^2$.
\end{cor}

\begin{proof}
Take the infimum over $Q\succ 0$ in \Href{Corollary}{cor:spectral}
and use the two identities in \Href{Proposition}{prp:duality}.
\end{proof}

For comparison, the estimate $S(C, Q) \geq \tr Q$ from
\Href{Proposition}{prp:ball-prodromou} gives the lower bound $d$ for
the left-hand side of \eqref{eq:combinatorial-dual}, by
\Href{Proposition}{prp:duality}. The improvement in
\Href{Corollary}{cor:combinatorial-dual} comes from using the one-sided
inequalities. The measure in that corollary can be chosen to maximize
$\tr\sqrt{A_\nu}$ over probability measures on the vertices.
\Href{Conjecture}{conj:gaussian-functional} asks for a bound for the
prescribed Gaussian vertex measure.

\section*{Acknowledgments}

The author thanks Fedor Nazarov for helpful discussions and for
suggesting several elegant steps in the proof.



\begin{thebibliography}{DLMLZ00}

\bibitem[AHAK23]{almendra2022quantitative}
V{\'\i}ctor~Hugo Almendra-Hern{\'a}ndez, Gergely Ambrus, and Matthew Kendall.
\newblock Quantitative {H}elly-type theorems via sparse approximation.
\newblock {\em Discrete \& Computational Geometry}, 70(4):1707--1714, 2023.

\bibitem[And03]{anderson2003introduction}
Theodore~W. Anderson.
\newblock {\em An {I}ntroduction to {M}ultivariate {S}tatistical {A}nalysis}.
\newblock Wiley Series in Probability and Statistics. John Wiley \& Sons,
  Hoboken, NJ, 3rd edition, 2003.

\bibitem[BL07]{bezdek2007vertex}
Karoly Bezdek and Alexander~E. Litvak.
\newblock On the vertex index of convex bodies.
\newblock {\em Advances in Mathematics}, 215(2):626--641, 2007.

\bibitem[B{\"o}r04]{boroczky2004finite}
K{\'a}roly B{\"o}r{\"o}czky, Jr.
\newblock {\em Finite packing and covering}, volume 154 of {\em Cambridge
  Tracts in Mathematics}.
\newblock Cambridge University Press, Cambridge, 2004.

\bibitem[BP09]{Ball2009}
K.~M. Ball and M.~Prodromou.
\newblock A sharp combinatorial version of {V}aaler{\textquotesingle}s theorem.
\newblock {\em Bulletin of the London Mathematical Society}, 41(5):853--858,
  August 2009.

\bibitem[DLMLZ00]{dalla2000blocking}
Leoni Dalla, David~G. Larman, Peter Mani-Levitska, and Chuanming Zong.
\newblock The blocking numbers of convex bodies.
\newblock {\em Discrete \& Computational Geometry}, 24(2):267--278, 2000.

\bibitem[FT72]{fejestoth1972lagerungen}
L{\'a}szl{\'o} Fejes~T{\'o}th.
\newblock {\em Lagerungen in der {E}bene auf der {K}ugel und im {R}aum},
  volume~65 of {\em Grundlehren der mathematischen Wissenschaften}.
\newblock Springer-Verlag, Berlin--Heidelberg, second edition, 1972.

\bibitem[GL12]{gluskin2012vertex}
Efim~D. Gluskin and Alexander~E. Litvak.
\newblock A remark on vertex index of the convex bodies.
\newblock In {\em Geometric Aspects of Functional Analysis: Israel Seminar
  2006--2010}, pages 255--265. Springer, 2012.

\bibitem[HJ12]{Horn2012}
Roger~A. Horn and Charles~R. Johnson.
\newblock {\em Matrix analysis}.
\newblock Cambridge University Press, 2012.

\bibitem[IN22]{ivanov2022quantitative}
Grigory Ivanov and M{\'a}rton Nasz{\'o}di.
\newblock A quantitative {H}elly-type theorem: containment in a homothet.
\newblock {\em SIAM Journal on Discrete Mathematics}, 36(2):951--957, 2022.

\bibitem[Iva21]{ivanov2021volume}
Grigory Ivanov.
\newblock On the volume of projections of the cross-polytope.
\newblock {\em Discrete Mathematics}, 344(5):112312, 2021.

\bibitem[Iva26]{ivanov2026maximalvolumeprojectionscrosspolytope}
Grigory Ivanov.
\newblock The maximal volume of projections of the cross-polytope.
\newblock \emph{arXiv preprint arXiv:2607.12072}, 2026.

\bibitem[Pr{\'e}71]{prekopa1971logarithmic}
Andr{\'a}s Pr{\'e}kopa.
\newblock Logarithmic concave measures with application to stochastic
  programming.
\newblock {\em Acta Scientiarum Mathematicarum}, 32:301--316, 1971.

\bibitem[Sam53]{sampford1953some}
Michael~R. Sampford.
\newblock Some inequalities on {M}ill's ratio and related functions.
\newblock {\em The Annals of Mathematical Statistics}, 24(1):130--132, 1953.

\bibitem[SZ12]{schmitt2012ten}
Moritz~W. Schmitt and G{\"u}nter~M. Ziegler.
\newblock Ten problems in geometry.
\newblock In Marjorie Senechal, editor, {\em Shaping Space: Exploring Polyhedra
  in Nature, Art, and the Geometrical Imagination}, pages 279--289. Springer,
  New York, 2012.

\bibitem[Tal61]{tallis1961moment}
Georges~M. Tallis.
\newblock The moment generating function of the truncated multi-normal
  distribution.
\newblock {\em Journal of the Royal Statistical Society Series B: Statistical
  Methodology}, 23(1):223--229, 1961.

\bibitem[Ver18]{vershynin2018high}
Roman Vershynin.
\newblock {\em High-dimensional probability: An introduction with applications
  in data science}, volume~47.
\newblock Cambridge University Press, 2018.

\bibitem[Zon05]{zong2005known}
Chuanming Zong.
\newblock What is known about unit cubes.
\newblock {\em Bulletin of the American Mathematical Society}, 42(2):181--211,
  2005.

\end{thebibliography}
\end{document}